\UseRawInputEncoding
\documentclass[a4paper]{article}
\pdfoutput=1
\usepackage[all]{xy}
\usepackage[latin1]{inputenc}        %accents
\usepackage[dvips]{graphics,graphicx}
\usepackage{amsfonts,amssymb,amsmath,color,mathrsfs, amstext}
\usepackage{amsbsy, amsopn, amscd, amsxtra, amsthm,authblk, enumerate}
\usepackage{enumerate,algorithmic,algorithm}
\usepackage{upref}
\usepackage[colorlinks,
            linkcolor=red,
            anchorcolor=red,
            citecolor=red
            ]{hyperref}

\usepackage{geometry}
\usepackage[displaymath]{lineno}
\usepackage{float}
\usepackage{yhmath}
\usepackage[normalem]{ulem}
\usepackage{subfig}

\numberwithin{equation}{section}     

\def\R{\mathbb{R}}
\def\b{\boldsymbol}

\def\T{\mathbb{T}}
\def\cL{\mathcal{L}}

\def\cA{\mathcal{A}}
\def\cR{\mathcal{R}}

\newcommand{\E}{\mathbb{E}}

\newtheorem{theorem}{Theorem}[section]
\newtheorem{lemma}{Lemma}[section]
\newtheorem{assumption}{Assumption}[section]
\newtheorem{proposition}{Proposition}[section]
\newtheorem{remark}{Remark}[section]
\newtheorem{corollary}{Corollary}[section]

\begin{document}

\title{Convergence of random splitting method for the Allen-Cahn equation in a background flow}
\author[a,b,c]{Lei Li\thanks{E-mail: leili2010@sjtu.edu.cn}}
\author[a]{Chen Wang\thanks{E-mail: wangchen6326@sjtu.edu.cn}}
\affil[a]{School of Mathematical Sciences, Shanghai Jiao Tong University, Shanghai, 200240, P.R.China.}
\affil[b]{Institute of Natural Sciences, MOE-LSC, Shanghai Jiao Tong University, Shanghai, 200240, P.R.China.}
\affil[c]{Shanghai Artificial Intelligence Laboratory}
\date{}
\maketitle

\begin{abstract}
We study in this paper the convergence of the random splitting method for Allen-Cahn equation in a background flow that plays as a simplified model for phase separation in multiphase flows. The model does not own the gradient flow structure as the usual Allen-Cahn equation does, and the random splitting method is advantageous due to its simplicity and better convergence rate. Though the random splitting is a classical method, the analysis of the convergence is not straightforward for this model due to the nonlinearity and unboundedness of the operators. We obtain uniform estimates of various Sobolev norms of the numerical solutions and the stability of the model. Based on the Sobolev estimates, the local trunction errors are then rigorously obtained. We then prove that the random operator splitting has an expected single run error with order $1.5$ and a bias with order $2$. Numerical experiments are then performed to confirm our theoretic findings.
\end{abstract}

\section{Introduction}\label{intro}

The Allen-Cahn model is among the most widely used phenomenological phase-field models for binary miscible fluids especially binary alloys \cite{allen1979microscopic}, and has been used in a range of problems such as crystal growth \cite{wheeler1992phase}, image analysis \cite{benevs2004geometrical}, mean curvature-flow \cite{feng2003numerical}, and random perturbations \cite{heida2018large}. 
Most phase-field models are derived as gradient flows associating with some specific energy functional and the classical Allen-Cahn model is the $L^2$ gradient flow of the Ginzburg-Landau energy, taking the standard form
\begin{equation}\label{equ:allen-cahn}
\partial_{t}u=\nu \Delta u-f(u),
\end{equation}
subject to initial data $u|_{t=0}=u_0$. Here, $(t,x)\in(0,\infty)\times\Omega$ for some domain $\Omega$, \begin{gather}\label{eq:nonlinearterm}
f(u)=u^3-u,
\end{gather}
 and $u$ corresponds to the concentration of a phase in the fluid \cite{allen1979microscopic}. The associated Ginzburg-Landau energy functional is given by
\begin{gather}
\mathcal{E}(u)=\int_{\Omega}\left[\frac{1}{2}\nu |\nabla u|^2+F(u)\right]\,dx,
\end{gather}
where $F'(u)=f(u)$ and $F(u)$ can be given by
\begin{gather}
F(u)=\frac{1}{4}(1-u^2)^2.
\end{gather}
Clearly, the equation can be formulated as
$\dot{u}=-\frac{\delta \mathcal{E}}{\delta u}$.
Other models include the Cahn-Hilliard equation \cite{cahn1958free} which is the $H^{-1}$ gradient flow of the Ginzburg-Landau energy and phase-field crystal models as the gradient flows of the Swift-Hohenberg energy  \cite{quan2022decreasing}.
%Gradient flows are frequently used in mathematical models for problems in many fields of science and engineering, particularly in materials science and fluid dynamics \cite{allen1979microscopic,anderson1998diffuse,cahn1958free,elder2002modeling,gurtin1996two,leslie1979theory,shen2018scalar,yue2004diffuse}. 

When one considers general multiphase flows, the phase field models should be coupled to fluid equations like the Navier-Stokes equations \cite{jacqmin1999calculation,praetorius2015navier}. As a simplified model, a background velocity field (e.g., a shear flow) could be added into the phase-field model as a linear convection term (see for example \cite{berthier2001phase,bray2003coarsening,o2007bubbles}). We are then motivated to consider the following Allen-Cahn equation with a background flow
\begin{equation}\label{equ:acdrift}
\partial_tu+v(x)\cdot\nabla u=\nu \Delta u-f(u),
\end{equation}
subject to initial data $u(x, 0)=u_0(x)$. For simplicity, we choose the domain $\Omega$ to be a $d$-dimensional torus with length $L$. The background field $v(\cdot)$ is a given time-independent flow with 
\begin{gather}
\nabla\cdot v(x)=0,
\end{gather}
modeling the incompressibility of the fluid. The nonlinear reaction term $f$ is the same as in \eqref{eq:nonlinearterm}.

The original Allen-Cahn equation \eqref{equ:allen-cahn} model can be numerically solved efficiently and stably by taking advantage of the gradient flow structure \cite{bottcher2017energy,shen2018scalar,liao2020energy,fu2024energy}. 
 Another class of numerical methods is the operator splitting method, which has been proposed by several authors to solve the Allen-Cahn model  \eqref{equ:allen-cahn} \cite{descombes2001convergence,weng2016analysis,li2022stability,li2022stability2}. 
The operator splitting method is a successful strategy to break a difficult problem into several relatively simpler subproblems \cite{strang1968construction,iserles2009first,glowinski2017splitting}. Combining the Strang splitting and the Richardson's extrapolation, the order of convergence could be improved to be at least third order. 
 The rigorous proof of the convergence of the splitting method for the Allen-Cahn equation is challenging due to the nonlinearity and the unboundedness, and has been resolved in \cite{li2022stability,li2022stability2}. In particular, the second order convergence of the Strang splitting has been established in \cite{li2022stability}.

With three physical terms--advection, diffusion and reaction, the model  
\eqref{equ:acdrift} does not own the gradient flow structure as  \eqref{equ:allen-cahn} does. We thus turn our attention to the splitting methods. Consider the problem in a general form
\begin{equation}\label{equ:pro-org}
\partial_tu=\sum_{i=1}^p \mathcal{L}_i(u)
\end{equation}
with initial data $u_0$. Here, some operators $\cL_i$ could be nonlinear. Generally, the structure of each $\mathcal{L}_i$ is much simpler than the sum. For each subproblem with initial data $w_0$
\begin{equation}\label{equ:pro-litt}
\partial_tw=\mathcal{L}_i(w),
\end{equation}
we define the solution semigroup $\{S_{i}(t)\}_{t\ge 0}$ (which could be nonlinear) by
\begin{equation}\label{equ:split-term-short-write-early}
S_{i}(t)[w_0]:=w(t).
\end{equation}
In the case $p=2$, with a time step $\tau>0$, the Strang splitting method is given by
\begin{equation}\label{equ:strang-split}
u_n=S_{1}(\tau/2)S_{2}(\tau)S_{1}(\tau/2)[u_{n-1}],
\end{equation}
and it has convergence order $2$. The Strang splitting  \cite{strang1968construction} for $p=2$ is efficient because one can combine the half-steps between consecutive time steps using the semigroup property so that essentially only two operators should be evaluated.  When $p\ge 3$, doing a Trotter splitting symmetrically as $S_{1}(\tau/2)\cdots S_{p-1}(\tau/2)S_{p}(\tau)S_{p-1}(\tau/2)\cdots S_{1}(\tau/2)[u_{n-1}]$ could also achieve second order, but now one needs to evaluate essentially $2p-2$ operators per time step, which looks inefficient. It turns out that the randomized methods \cite{Novak1988,jin2020random,2011On,Cao2021} and the randomized operator splitting methods \cite{eisenmann2024randomized} could be beneficial due to the simplicity and the averaging effect in time. The random operator splitting method of Trotter type applied to the general problem  \eqref{equ:pro-org} is then given by 
\begin{equation}\label{eq:randomproduct}
u_{n}=S_{\xi_p}(\tau)\cdots S_{\xi_1}(\tau)[u_{n-1}]=:\prod_{i=1}^pS_{\xi_i}(\tau)[u_{n-1}],
\end{equation}
where $(\xi_1,\cdots, \xi_p)$ is a random permutation of $\{1,2,\cdots, p\}$ that is independent for every $n$ and $\tau$ is time step. Clearly, one only needs to evaluate $p$ operators per time step, with the same cost as the deterministic ones. However, due to the averaging effect in time evolution (see the related discussion in \cite{jin2020random,li2024random}), the order of convergence is better than first order which is the typical order for deterministic splitting. For theoretical analysis, it is straightforward to justify the order of convergence for splitting methods with linear operators, using the Baker-Campbell-Hausdorff (BCH) formula \cite{hairer2006geometric}. However, for the problem considered, there are both unbounded and nonlinear operators and the analysis is not straightforward.

In this paper, we consider the convergence analysis of the aforementioned random splitting method applied to \eqref{equ:acdrift}.  Motivated by \cite{li2022stability,li2022stability2}, we would first estimate the uniform Sobolev bounds and then establish the stability. Although the three terms in the convected Allen-Cahn equation are not bounded, the action on the numerical solutions can be bounded due to the Sobolev estimates. This allows us to compute the local truncation errors in detail, although there is nonlinear dynamics. Together with the stability of the dynamics of the model itself, we are then able to control the desired numerical bias and expected single run error.
The main results can be summarized as (see Theorem \ref{theo:4-2} and Theorem \ref{theo:4-3} respectively)
\begin{equation}
\mathbb{E}_{\xi}\Vert u_n-u(n\tau)\Vert_{k,p}\leq C\tau^{3/2}
\end{equation}
and
\begin{equation}
\Vert \mathbb{E}_{\xi}u_n-u(n\tau)\Vert_{k,p}\leq C\tau^{2}.
\end{equation}
Note that the expected single run error is proved for $p<\infty$ only.

We remark that the operator splitting or particularly Lie-Trotter splitting have been widely applied for various models, including the stochastic differential equations  \cite{abdulle2015long,buckwar2022splitting} and particularly for Langevin sampling \cite{leimkuhler2013rational,leimkuhler2013robust,li2024second}, as well as other models like the Schr\"odinger type equations \cite{bao2002time,antoine2013computational,zhang2024low}.

The rest of the paper is organized as follows.  
In Section \ref{sec:problem}, we introduce the basic notations and setup for the problem and the method. In Section \ref{sec:boundness}, we establish the boundness and stability for the exact solution and the numerical solution. The propagation of the Sobolev regularity of the numerical solution is crucial for the analysis later. In Section \ref{sec:error}, we analyze the expected single-run error and the bias of the random splitting method. An illustrating numerical example is presented in Section \ref{sec:numerical}, confirming our theoretical results and demonstrating the performance of the random splitting method.  

\section{Setup and notations}\label{sec:problem}

As mentioned above, we will consider the model \eqref{equ:acdrift}
on a torus with length $L$, i.e.,
\begin{gather}
\Omega=\T^d,
\end{gather}
where $\T^d=[0, L]^d$ equipped with the periodic boundary conditions.
Here, $d\ge 2$ is the dimension.

The background velocity is assumed to be a smooth flow and divergence free for the incompressibility. 
\begin{assumption}\label{ass:vassumption}
The velocity field $v(\cdot)$ is infinitely smooth and satisfies
\begin{equation}
\mathrm{div}(v(x))=\nabla\cdot v(x)=0.
\end{equation}
\end{assumption}
Moreover, we also consider smooth initial data in this work for the numerical analysis. 

\subsection{The random splitting method for the model}

For the splitting method of \eqref{equ:acdrift}, we introduce the following notations.
\begin{equation}
\begin{split}
&\mathcal{L}_1u=\mathcal{A}u:=- v(x)\cdot\nabla u,\\
&\mathcal{L}_2u=\mathcal{L}u:=\nu \Delta u ,\\
&\mathcal{L}_3u=\mathcal{R}u:=-f(u).
\end{split}
\end{equation}
The notations $\mathcal{A},\mathcal{L},\mathcal{R}$ represent the advection drift term, the linear diffusion term (Laplacian), and the nonlinear reaction term respectively.
The three subproblems corresponding to \eqref{equ:pro-litt} are given by
the advection equation
\begin{gather}\label{subeq:advec}
\partial_t w=\cA w=-v\cdot\nabla w, \quad w|_{t=0}=w_0,
\end{gather}
the heat equation
\begin{gather}\label{subeq:heat}
\partial_t w=\cL w=\nu \Delta w, \quad w|_{t=0}=w_0,
\end{gather}
and the following reaction equation
\begin{gather}\label{subeq:react}
\partial_t w=\cR w= -f(w), \quad w|_{t=0}=w_0,
\end{gather}
which is essentially an ODE.

The semigroup $\{S_1(t)\}_{t\ge 0}$ is the solution to the advection equation, which can essentially solved by the method of characteristics. We will not, however, use the method of characteristics explicitly in this paper, and will use the property of the PDE directly. In fact, the advection term is antisymmetric and $S_1$ is unitary in $L^2(\Omega)$. Moreover, any $L^p$ norm is preserved.

The semigroup $\{S_2(t)\}_{t\ge 0}$ corresponding to $\cL_2=\nu \Delta$ is simply the heat semigroup \(S_2(t)=\exp(\nu t\Delta )\), which is well-known to be dissipative in various spaces \cite{evans2022partial}. 

The semigroup $\{S_3(t)\}_{t\ge 0}$ corresponding to $\cL_3=\cR$ is the solution to the ODE
\[
\int \frac{dw}{f(w)}=-\int dt.
\]
In the case \eqref{eq:nonlinearterm}, it is given by
\begin{gather}\label{eq:nonlinearsemigroup}
w(t)=S_3(t)[w_0]=\frac{w_0}{\sqrt{w_0^2+(1-w_0^2)e^{-2t}}}.
\end{gather}
This formula will be used in the numerical experiment in section \ref{sec:numerical}. We will, however, not use this explicit formula in the analysis in section \ref{sec:error}. Instead, we will use the equation to estimate various norms for $S_3$, so that the derivation could be generalized to other forms of $f$.

With the semigroups introduced, the random splitting method can be formally formulated as an algorithm shown in Algorithm \ref{alg:randomsplit}.
\begin{algorithm}[!ht]
    \caption{The Random splitting method for the Allen-Cahn model in a flow}
    \label{alg:randomsplit}
    \begin{algorithmic}[1]
        \REQUIRE 
        Time step $\tau>0$, initial data $u_0$, terminal time $T>0$.
        
        \FOR{$n=0: \lceil T/\tau\rceil-1$}
        \STATE Generate a random permutation $\xi^{n}$ of $\{1,2,3\}$ independently of the previous ones and set
        \begin{gather}\label{alog-equ:random-split-term}
        u_{n+1}=\prod_{i=1}^3S_{\xi_i^{n}}(\tau)[u_{n}].
        \end{gather}
        \ENDFOR
    \end{algorithmic}
\end{algorithm}

For later analysis, we introduce some notation. Define the time grid
\begin{gather}
t_n=n\tau,
\end{gather}
where $\tau>0$ is the timestep. Then, $I_n:=[t_{n},t_{n+1})$ is the $n$th subinterval. Let
\[
\b{\xi}:=(\xi^0,\cdots, \xi^{N-1})
\]
be the sequence of random permutations, where $N=\lceil T/\tau\rceil$ is the number of time steps in the computation and
\[
\xi^{n}=(\xi_i^{n})_{i=1}^3
\]
is the random permutation for the subinterval $I_n$.
We then define the random splitting semigroup
\begin{equation}\label{equ:random-split-term}
S^{\xi^n}(\tau):=\prod_{i=1}^3S_{\xi_i^n}(\tau),
\end{equation}
and the mean of this random splitting semigroup
\begin{equation}\label{equ:ex-random-split-term}
S(\tau):=\mathbb{E}_{\xi^n}S^{\xi^n}(\tau).
\end{equation}
Note that the convention in \eqref{eq:randomproduct} is applied for \eqref{alog-equ:random-split-term} and \eqref{equ:random-split-term}.
The definition of $S$ in \eqref{equ:ex-random-split-term} is independent of $n$ since the equation is time-homogeneous.  Moreover, for the model \eqref{equ:acdrift} itself, we introduce its time-continuous semigroup $T(\tau)$ given by
\begin{equation}\label{equ:exact-solution-operator}
T(\tau)u_0:=u(\tau).
\end{equation}

With the above notations, the numerical solution is then given by
\begin{equation}\label{rela-rec}
u_{n+1}=S^{\xi^n}(\tau)[u_{n}],
\end{equation}
and the exact solution $u(\cdot)$ satisfies 
\begin{gather}
u(t_{n+1})=T(\tau)[u(t_n)].
\end{gather}
We also denote the expected numerical solution as $v_n$ so that 
\begin{equation}\label{def:v-n}
v_{n}=\mathbb{E}_{\b{\xi}}u_{n}.
\end{equation}

\subsection{Other notations}

One denotes for any $p\in [1,\infty)$ the standard $L^p$ space by
\[
L^p(\Omega):=\left\{u: \int_{\Omega} |u|^p \,dx<\infty\right\}
\] 
and the corresponding norm is given by
\[
\|u\|_p:=\left(\int_{\Omega} |u|^p \,dx\right)^{1/p}.
\]
For $p=\infty$, one defines
\[
\|u\|_{\infty}=\mathrm{esssup}\{|u(x)|: x\in \Omega\},
\]
which means the smallest constant $M$ such that the Lebesgue measure of $\{x: |u(x)|>M\}$ is zero. As well-known, $\|u\|_{\infty}=\lim_{p\to\infty}\|u\|_p$.

A multi-index for $\R^d$ is a tuple $\alpha=(\alpha_1,\cdots, \alpha_d)$
with $\alpha_i\in \mathbb{N}=\{0,1,2,\cdots\}$. Then, $|\alpha|=\sum_{i=1}^d \alpha_i$ is called the order of $\alpha$, and we will use the following notation of multi-derivatives
\begin{gather}
D^{\alpha}u=\prod_{i=1}^d\left(\frac{\partial}{\partial x_i}\right)^{\alpha_i}u.
\end{gather}

We will also introduce the following
\begin{equation}\label{equ:a}
\Gamma_{\alpha, m}=\begin{cases} \Big\{ \{\beta_1,\beta_2,\cdots,\beta_m\} |\sum_{i=1}^{m}\beta_i=\alpha,|\beta_i|\geqslant 1 \Big\}  &  m\le |\alpha|,  \\
\emptyset & m>|\alpha|.
\end{cases}
\end{equation}
Note that in \eqref{equ:a}, the notation $\{\beta_1,\beta_2,\cdots,\beta_m\}$ means a set of multi-indices so it remains unchanged if the order of the multi-indices inside changes. 

We will provide error analysis in various Sobolev spaces. Recall that 
the $W^{k, p}$ norm is given by
\begin{equation}
\Vert u\Vert_{k,p}=\bigg(\sum_{|\alpha|\leq k}\Vert D^{\alpha}u\Vert_{p}^{p}\bigg)^{1/p},
\end{equation}
where $D^{\alpha}u$ is the weak derivative of the function $u$ (see \cite{evans2022partial}). Then, $W^{k,p}$ space consists of measurable functions that has finite $\|\cdot\|_{k,p}$ norm. 
When $p=+\infty$,  the $\|\cdot\|_{k,p}$ reduces to the following
\begin{equation}
\Vert u\Vert_{k,\infty}=\max_{|\alpha|\leq k}\Vert D^{\alpha}u\Vert_{\infty}.
\end{equation}
Clearly, $L^p(\Omega)=W^{0,p}(\Omega)$, and $\|u\|_p=\|u\|_{0,p}$.

\section{Boundedness and Stability}\label{sec:boundness}

In this section, we establish the boundedness of the solutions, especially the numerical solutions, and the stability of the PDE model.
These are essential to establish the global error estimates of the random splitting method when it is applied to this model.

\subsection{Boundedness of Sobolev norms}

In this subsection, we establish the boundedness of the Sobolev norms with various orders, both for the solutions of the PDE model and for the solutions of the random splitting method. These results would be useful for establishing the stability of the model and the estimate of local truncation errors.

\subsection*{Sobolev norms for the PDE model}

We first establish theorem about the bounds for the PDE model, which is summarized as following.
\begin{proposition}\label{pro:exactsol}
Consider problem \eqref{equ:acdrift} on $\Omega$ and 
suppose that $v$ satisfies Assumption \ref{ass:vassumption}.  
Then for any given $p\in[2,+\infty]$ and $k\in \mathbb{N}$, if $\|u_0\|_{k,p}+\|u_0\|_{\max(k-1,0),\infty}<\infty$, one then has
\begin{equation}
\sup_{t\le T}\Vert u(\cdot, t)\Vert_{k,p}<\infty.
\end{equation}
\end{proposition}

\begin{proof}
The proof is a straightforward energy estimate together with the induction. 
First, the solution $u$ is smooth on the interval of existence so that all the derivatives taken below can be justified. This said, direct computation gives the following expression for the time derivative of the $p$th power of the $L^p$ norm of the partial derivatives.
\begin{equation}\label{iden:1}
\begin{split}
\frac{d}{d t}\Vert D^{\alpha}u\Vert_p^p&= p\int_{\Omega}|D^{\alpha}u|^{p-2}D^{\alpha}uD^{\alpha}u_tdx=\nu p\int_{\Omega}|D^{\alpha}u|^{p-2}D^{\alpha}u\Delta(D^{\alpha}u)dx\\
&\quad -p\int_{\Omega}|D^{\alpha}u|^{p-2}D^{\alpha}uD^{\alpha}f(u)dx-p\int_{\Omega}|D^{\alpha}u|^{p-2}D^{\alpha}uD^{\alpha}(\nabla\cdot(vu))dx\\
&=:I_1+I_2+I_3.
\end{split}
\end{equation}

Let $k=|\alpha|$. We first show the results for $k=0$. The relation \eqref{iden:1} becomes
\begin{equation*}
\begin{split}
\frac{d}{d t}\Vert u\Vert_p^p &=\nu p\int_{\Omega}|u|^{p-2}u\Delta udx+p\int_{\Omega}|u|^{p-2}u(u-u^3)dx-p\int_{\Omega}|u|^{p-2}u\nabla u\cdot v(x)dx\\
&=-\nu p(p-1)\int_{\Omega}|u|^{p-2}|\nabla u|^2dx
+p\int_{\Omega}(|u|^{p}-|u|^{p+2})dx,
\end{split}
\end{equation*}
where we have used the fact that
\begin{equation}\label{case0:i3}
I_3=-\int_{\Omega}v(x)\cdot\nabla|u|^pdx=\int_{\Omega}(\nabla\cdot v)|u|^pdx=0.
\end{equation}
Then,
\begin{equation*}
\frac{d}{d t}\Vert u\Vert_p^p\leq p\Vert u\Vert_p^p.
\end{equation*}
This gives that
\[
\Vert u(\cdot,t)\Vert_p\leq e^{t}\Vert u_0\Vert_p\leq e^{T}\Vert u_0\Vert_p.
\]
Note that the prefactor is independent of $p$ and thus sending $p\to\infty$ gives the result for $p=\infty$ as well. Hence, the claim holds for $k=0$.

For $k\ge 1$, we perform the argument by induction. Assume the results holds for all $|\alpha|\le k-1$. For $|\alpha|=k$, consider \eqref{iden:1}. The $I_1$ term can be simply estimated by integration by parts.
\begin{equation}\label{inequ:i1}
\begin{split}
I_1 &=-\nu p\int_{\Omega}\nabla(|D^{\alpha}u|^{p-2}D^{\alpha}u)\cdot\nabla(D^{\alpha}u)dx\\
 &=-\nu p(p-1)\int_{\Omega}|D^{\alpha}u|^{p-2}|\nabla(D^{\alpha}u)|^2dx\leq 0.
\end{split}
\end{equation}

Recall that
\[
I_2=-p\int_{\Omega}|D^{\alpha}u|^{p-2}D^{\alpha}uD^{\alpha}f(u)dx.
\]
To compute this, we recall \eqref{equ:a} for $\Gamma_{\alpha, m}$.
Then, it is not hard to see that
\begin{equation}\label{equ:3-4}
D^{\alpha}(f(u))=\sum_{m=1}^{|\alpha|}\sum_{\{\beta_1,\beta_2,\cdots,\beta_m\}\in\Gamma_{\alpha,m}}C_{\alpha,(\beta_1,\beta_2,\cdots,\beta_m)}f^{(m)}(u)\cdot D^{\beta_1}u\cdot D^{\beta_2}u\cdots D^{\beta_m}u.
\end{equation}
Here, $C_{\alpha,\{\beta_1,\beta_2,\cdots,\beta_m\}}$ is some positive integer depending on the set of indices, and it satisfies
\begin{equation*}
\sum_{(\beta_1,\beta_2,\cdots,\beta_m)\in\Gamma_{\alpha,m}}C_{\alpha,(\beta_1,\beta_2,\cdots,\beta_m)}\le k!.
\end{equation*}
If $m\ge 2$, all the $\beta_i$ involved would have lower order than $\alpha$. Hence, one has
\begin{multline*}
I_2=p\int_{\Omega}|D^{\alpha}u|^{p}dx-3p\int_{\Omega}u^2|D^{\alpha}u|^{p}dx
-6p\int_{\Omega}|D^{\alpha}u|^{p-2}D^{\alpha}u \\
\left(\sum_{(\beta_1,\beta_2)\in\Gamma_{\alpha,2}}C_{\alpha,(\beta_{1},\beta_{2})}uD^{\beta_{1}}uD^{\beta_{2}}u 
+\sum_{(\beta_1,\beta_2,\beta_{3})\in\Gamma_{\alpha,3}}C_{\alpha,(\beta_{1},\beta_{2},\beta_{3})}D^{\beta_{1}}uD^{\beta_{2}}uD^{\beta_{3}}u\right)   dx.
\end{multline*}
As we mentioned, all the derivatives in the parenthesis have lower orders than $\alpha$. Hence, one has
\begin{equation}\label{inequ:i2}
\begin{split}
I_2 &\le  p\Vert D^{\alpha}u\Vert_p^p+pC\Vert u\Vert_{|\alpha|-1,\infty}^2\sum_{\beta_1} \int_{\Omega}|D^{\alpha}u|^{p-1}|D^{\beta_1}u|dx\\
& \le p\Vert D^{\alpha}u\Vert_p^p+pC\Vert u\Vert_{k-1,\infty}^2\Vert u\Vert_{k-1,p}  \Vert D^{\alpha}u\Vert_{p}^{p-1}.
\end{split}
\end{equation}

For $I_3$, we first note that by Leibniz formula  that
\begin{equation*}
D^{\alpha}\nabla\cdot(vu)=\sum_{\beta: \beta\leq \alpha}C_{\alpha}^{\beta}D^{\beta}v\cdot\nabla(D^{\alpha-\beta}u).
\end{equation*}
The $\beta=0$ term contains the highest derivative but integration by parts and the divergence-free condition of $v$ gives that
\begin{equation*}
p\int_{\Omega}|D^{\alpha}u|^{p-2}D^{\alpha}u\nabla(D^{\alpha}u)\cdot vdx=\int_{\Omega}v\cdot\nabla|D^{\alpha}u|^pdx=0.
\end{equation*}
Consequently, since $v$ is smooth, one obtains
\begin{equation}\label{inequ:i3}
\begin{split}
I_3 &=-p\sum_{\beta\leq\alpha,\beta\neq 0}C_{\alpha}^{\beta}\int_{\Omega}|D^{\alpha}u|^{p-2}D^{\alpha}uD^{\beta}v\cdot\nabla(D^{\alpha-\beta}u)dx\\
&\le p C \sum_{|\alpha'| \le |\alpha|}\int |D^{\alpha}u|^{p-1}
|D^{\alpha'}u|\,dx
\le pC\Vert D^{\alpha}u\Vert_p^{p-1}\Vert u\Vert_{k,p}.
\end{split}
\end{equation}
Here $C$ is a constant that depends on $k$ and $d$ but is independent of $p$. 

Taking the sum of formula \eqref{iden:1} over $|\alpha|\leq k$, one then has:
\begin{equation*}
\frac{d}{d t}\Vert u\Vert_{k,p}^p\leq pC\Vert u\Vert_{k,p}^p+pC\Vert u\Vert_{k,p}^{p-1}\Vert u\Vert_{k-1,\infty}^2\Vert u\Vert_{k-1,p},
\end{equation*}
which implies that
\begin{equation*}
\frac{d}{d t}\Vert u\Vert_{k,p}\leq C\Vert u\Vert_{k,p}+C\Vert u\Vert_{k-1,\infty}^2\Vert u\Vert_{k-1,p},
\end{equation*}
where $C$ is independent of $p$.
 
Combining with the induction hypothesis, applying the Gr\"onwall's inequality gives the result for $k$. Note that since the coefficient is independent of $p$, the result holds for $p=\infty$ as well. The proof is complete.
\end{proof}

\subsection*{Sobolev norms for the random splitting method}

Next, we show that the numerical solutions of the random splitting method are bounded in Sobolev spaces.

We first summarize the properties of the semigroups in several lemmas.
\begin{lemma}\label{lmm:adv}
Consider the semigroup $S_1(t)$ corresponding to \eqref{subeq:advec} for the advection term. Then, $S_1$ is unitary in $L^2$ and for any $p\in [1,\infty]$,
\begin{gather}
\|S_1(t)[w_0]\|_p=\|w_0\|_p.
\end{gather}
Moreover, for any $p\in [2,\infty]$ and $k\ge 0$, there is a constant $C$ such that
\begin{equation}
\Vert S_{1}(t)[w_0]\Vert_{k,p}\leq e^{Ct}\Vert w_0\Vert_{k,p}.
\end{equation}
\end{lemma}
\begin{proof}
Recall the corresponding equation for $w(t):=S_1(t)[w_0]$:
\begin{equation*}
\partial_t w=-v(x)\cdot\nabla w
\end{equation*}
with $w(0)=w_0$.
For any $\varphi_0\in L^2$ and the corresponding $\varphi(t)=S_1(t)[\varphi_0]$, one has
\[
\begin{split}
\frac{d}{dt}\int w(t)\varphi(t)\,dx &=\int \left((-v\cdot\nabla w(t)) \varphi(t)
+w(t)(-v\cdot\nabla \varphi(t))\right)\,dx\\
&=\int \left((\nabla\cdot v) w(t) \varphi(t)+w(t)v\cdot\nabla\varphi(t)
+w(t)(-v\cdot\nabla \varphi(t))\right)\,dx=0.
\end{split}
\]
Hence, the semigroup is unitary and thus $\|w(t)\|_2$ is unchanged.

For any $p\in (1,\infty)$, one can computes
\[
\frac{d}{dt}\int_{\Omega} |w|^p\,dx=-p\int_{\Omega}|w|^{p-2}w\nabla w\cdot v(x)dx
=-\int_{\Omega} v\cdot\nabla |v|^p\,dx=\int_{\Omega} |v|^p(\nabla\cdot v)\,dx=0.
\]
Taking the limit $p\to 1^+$ gives the result for $p=1$, which can also be justified by considering the regularized quantity $\int\sqrt{\epsilon+|v|^2}\,dx$ for $L^1$ norm.

Next for any $p\in[2,\infty)$ and $k\in \mathbb{N}$, by the same calculation as in the proof of Proposition \ref{pro:exactsol}, and particularly by \eqref{inequ:i3}, one has
\begin{gather}
\frac{d}{dt}\Vert D^{\alpha}u\Vert_p^p \le pC\Vert D^{\alpha}u\Vert_p^{p-1}\Vert u\Vert_{k,p}.
\end{gather}
Summing over $|\alpha|\leq k$ gives
\[
\frac{d}{dt}\Vert u\Vert_{k,p}^p \le pC\Vert u\Vert_{k,p}^{p-1}\Vert u\Vert_{k,p} \quad
\Longrightarrow  \quad \frac{d}{d t}\Vert u\Vert_{k,p}\leq C\Vert u\Vert_{k,p}.
\]
This gives the result for $p\in [2,\infty)$.
Note that the coefficient $C$ is independent of $p$ so that the claims about $S_1$ hold for all $p\in [2,\infty]$.
\end{proof}

The following result for heat semigroup is well-known in PDE theory but we still attach a simple proof for the convenience of readers.
\begin{lemma}
Consider the heat semigroup $S_2(t)$ for $\nu \Delta$ corresponding to 
\eqref{subeq:heat}. Then, for any $p\in [2,\infty]$ and $k\ge 0$, one has
\begin{equation}
\Vert S_{2}(t)[w_0]\Vert_{k,p}\leq \Vert w_0\Vert_{k,p}.
\end{equation}
\end{lemma}

\begin{proof}
Consider the associated problem with initial value $w(0)=w_0$:
\begin{equation*}
\partial_t w=\mathcal{L}_2 u=\nu \Delta w.
\end{equation*}
Then we can compute that,
\begin{equation*}
\begin{split}
\frac{\partial}{\partial t}\Vert D^{\alpha}u\Vert_{p}^p&=\nu p\int_{\Omega}|D^{\alpha}u|^{p-2} D^{\alpha}u\Delta(D^{\alpha}u)dx\\
&=-\nu p(p-1)\int_{\Omega}|D^{\alpha}u|^{p-2}|\nabla(D^{\alpha}u)|^2dx\leq  0.
\end{split}
\end{equation*}
Hence, the Sobolev norms are nonincreasing for all $p\ge 2$.
Let $p\rightarrow+\infty$, we get the result for $\Vert\cdot \Vert_{k,\infty}$.
\end{proof}

Last, we consider the nonlinear reaction semigroup.
\begin{lemma}\label{lmm:react}
Consider the semigroup $S_3(t)$ for the nonlinear reaction term corresponding to \eqref{subeq:react}.  Then for all $p\in[2,+\infty]$ and $k\in\mathbb{N}_+$, there exists a polynomial $p_k(\cdot)$ such that
\begin{equation}
\Vert S_{3}(t)[w_0]\Vert_{k,p}\leq  e^{t}\Vert w_0\Vert_{k,p}+p_k(\Vert w_0\Vert_{k-1,\infty}, \Vert w_0\Vert_{k-1,p}, e^t)(e^t-1).
\end{equation}
If $k=0$, the polynomial does not depend on $\|w_0\|_{k-1,\infty}$.
\end{lemma}

\begin{proof}
Consider the nonlinear problem
\begin{equation*}
\partial_t w=-f(w)=w-w^3.
\end{equation*}
First consider the case $k=0$.
\begin{equation*}
\frac{d}{d t}\Vert w\Vert_p^p=p\int_{\Omega}|w|^{p}dx-p\int_{\Omega}|w|^{p+2}dx\leq p\Vert w\Vert_{p}^p.
\end{equation*}
Then it follows that
\begin{equation}\label{case0:i2}
\Vert S_{3}(t)[w_0]\Vert_{p}\leq e^{t}\Vert w_0\Vert_{p}.
\end{equation}
Hence, the claim for $k=0$ holds and we can take $p_0(z_1, z_2, z_3)=0$.

Following the same calculation as that in the proof of Proposition \ref{pro:exactsol}, and in particular the derivation of \eqref{inequ:i2}, one has 
\begin{equation*}
\frac{d}{d t}\Vert u\Vert_{k,p}^p    \leq p\Vert u\Vert_p^p+pC\Vert u\Vert_{k-1,\infty}^2\Vert u\Vert_{k-1,p}  \Vert u\Vert_{k,p}^{p-1},
\end{equation*}
and thus
\begin{equation}\label{inequ:part2}
\frac{d}{d t}\Vert u\Vert_{k,p}\leq \Vert u\Vert_{k,p}+C\Vert u\Vert_{k-1,\infty}^2\Vert u\Vert_{k-1,p}.
\end{equation}
Here, $C$ is independent of $p$. By induction, it is easy to see that the polynomial $p_k$ exists and the desired result holds. 

Moreover, since the coefficient is independent of $p$, the case for $p=\infty$ also holds.
\end{proof}

After proving the results about the semigroups, we then can conclude the most important result in this section, i.e., the propagation of the Sobolev regularity of the numerical solution for the random splitting method.
\begin{theorem}\label{theo:3-2}
Consider problem \eqref{equ:acdrift} on $\Omega$ and suppose that $v$ satisfies Assumption \ref{ass:vassumption}. Consider the numerical solution of Algorithm \ref{alg:randomsplit} given in \eqref{rela-rec}.
Let $T>0$, $p\in[2,+\infty]$ and $k\in \mathbb{N}$. We assume $\|u_0\|_{k,p}+\|u_0\|_{k-1,\infty}<\infty$ if $k\ge 1$ and $\|u_0\|_{k,p}<\infty$ if $k=0$. Then there is a constant $M_{k,p}$ independent of $n$, $\tau$ and the sequence of random permutation $\b{\xi}$ such that
\begin{equation}\label{inequ:num-solu-1}
\sup_{n: (n+1)\tau\le T}\left(\Vert S_{\xi^n_1}(\tau)u_{n}\Vert_{k,p} +\Vert S_{\xi_2^n}(\tau)S_{\xi_1^n}(\tau)u_{n}\Vert_{k,p}+\Vert u_{n+1}\Vert_{k,p}\right) \le M_{k,p}.
\end{equation}
Moreover, the constant $M_{k,p}$ can be chosen to be bounded as $p\to\infty$.
\end{theorem}
\begin{proof}
First consider $k=0$.
By Lemma \ref{lmm:adv}-\ref{lmm:react}, one has for any $p\in [2,\infty]$
\begin{equation*}
\begin{split}
& \Vert S_1(\tau)w_0\Vert_p=\Vert w_0\Vert_p,\\
& \Vert S_{2}(\tau)w_0\Vert_p\le \Vert w_0\Vert_p,\\
& \Vert S_{3}(\tau)w_0\Vert_p\le  e^{\tau}\Vert w_0\Vert_p.
\end{split}
\end{equation*}
Then for any $\b{\xi}$ and any $n$, $\Vert u_n\Vert_p\leq e^{\tau}\Vert u_{n-1}\Vert_p$. Then, it is evident that one can choose $M_{0,p}=e^T\Vert u_0\Vert_{p}$.

For $k=\ell$ with $\ell\ge 1$, we do induction and assume that the claims hold for $k\le \ell-1$.
By Lemma Lemma \ref{lmm:adv}-\ref{lmm:react}, for all $p\in[2,+\infty]$, $k\in\mathbb{N}$, one has for $w_0=u_n$ for some $n$ that
\begin{equation}\label{auxeq:semigroup}
\begin{split}
&\Vert S_1(\tau)w_0\Vert_{\ell,p}\leq e^{C\tau}\Vert w_0\Vert_{\ell,p},\\
&\Vert S_{2_1}(\tau)w_0\Vert_{\ell,p}\leq \Vert w_0\Vert_{\ell,p}, \\
& \Vert S_{3}(\tau)w_0\Vert_{\ell,p}\leq e^{\tau}\Vert w_0\Vert_{\ell,p}+\tilde{M}_{\ell,p}(e^{\tau}-1),\\
\end{split}
\end{equation}
where $\tilde{M}_{k,p}$ depends on the constants for $k\le \ell-1$ and $T$, but independent of $n$, $\tau$ and $\b{\xi}$.
Then, it holds that
\[
\|u_{n+1}\|_{\ell, p}\le e^{(C+1)\tau}\|u_n\|_{\ell, p}+\tilde{M}_{\ell,p}e^{C\tau}(e^{\tau}-1).
\]
A simple induction gives the desired result for $u_n$.
For the intermediate solutions $S_{\xi^n_1}(\tau)u_{n}$ and $S_{\xi_2^n}(\tau)S_{\xi_1^n}(\tau)u_{n}$, using the properties in \eqref{auxeq:semigroup}, the result then follows.
\end{proof}

\subsection{Stability of the Model}

In this subsection, we establish the stability of the original PDE model. Namely, when the initial values are close, the solutions stay close. 
The result is the following.
\begin{theorem}\label{theo:3-3}
Consider the model \eqref{equ:acdrift} and suppose Assumption \ref{ass:vassumption} holds.
Let $u_i(t)$ ($i=1,2$) be two solutions with initial data $u_i(0)$. Let $p\in [2,\infty]$, $k\in\mathbb{N}$, and we assume 
\begin{gather}
u_i(0)\in W^{k,p}\cap W^{\max(k-1,0),\infty}.
\end{gather}
Then there exists $C>0$ that could depend on the initial value but independent of $p$ such that
\begin{equation}
\Vert u_1(\tau)-u_2(\tau)\Vert_{k,p}\le e^{C\tau}\Vert u_1(0)-u_2(0)\Vert_{k,p}.
\end{equation}
\end{theorem}
\begin{proof}

First consider $k=0$, namely the stability in $L^p(\T^d)$. By the result in 
Proposition \ref{pro:exactsol}, $u_i(\cdot)\in L^{\infty}(0,T; L^{\infty})$.
 Taking the difference between the equations for $u_1$ and $u_2$, one obtains that
\begin{equation*}
\partial_t(u_1-u_2)+v\cdot\nabla(u_1-u_2)=\nu \Delta(u_1-u_2)-[f(u_1)-f(u_2)].
\end{equation*}
Then, one can compute directly that
\begin{equation*}
\begin{split}
\frac{d}{d t}\Vert u_1-u_2\Vert_p^p 
&=p\nu \int_{\Omega} |u_1-u_2|^{p-2}(u_1-u_2)\Delta(u_1-u_2)dx\\
&\quad -p\int_{\Omega} |u_1-u_2|^{p-2}(u_1-u_2)[f(u_1)-f(u_2)]dx\\
&\quad -p\int_{\Omega} |u_1-u_2|^{p-2}(u_1-u_2) v\cdot\nabla(u_1-u_2)dx
=: I_1+I_2+I_3.
\end{split}
\end{equation*}

It is straightforward to find that
\begin{equation*}
\begin{split}
I_1 &=-p\nu\int_{\Omega} \nabla(|u_1-u_2|^{p-2}(u_1-u_2))\nabla(u_1-u_2)dx\\
&=-\nu p(p-1)\int_{\Omega}|u_1-u_2|^{p-2}|\nabla(u_1-u_2)|^2dx\le 0
\end{split}
\end{equation*}
and that
\begin{equation*}
I_3=-p\int_{\Omega} |u_1-u_2|^{p-2}(u_1-u_2)v(x)\cdot\nabla(u_1-u_2)dx=-\int_{\Omega}v(x)\cdot\nabla|u_1-u_2|^{p}dx=0,
\end{equation*}
due to the fact that $\nabla\cdot v(x)=0$.

Next, we estimate $I_2$.
Noting that
\begin{equation*}
f(u_1)-f(u_2)=\int_{0}^{1}f'(u_2+t(u_1-u_2))(u_1-u_2)\,dt
\end{equation*}
and that
\begin{equation*}
\bigg|\int_{0}^{1}f'(u_2+t(u_1-u_2))dt\bigg|\le C
\end{equation*}
for some $C$ independent of $p$
according to Proposition \ref{pro:exactsol}, one then finds that
\begin{equation*}
\frac{d}{d t}\Vert u_1-u_2\Vert_p^p\le  Cp\Vert u_1-u_2\Vert_p^p.
\end{equation*}
Hence,
\[
\Vert u_1(t)-u_2(t)\Vert_p \le 
e^{Ct}\Vert u_1(0)-u_2(0)\Vert_p.
\]
Since $C$ is independent of $p$, taking $p\to\infty$ then gives the estimate for $L^{\infty}$ norm.

Now, we move to the cases for $k=\ell\ge 1$. We will do the estimates by induction. Suppose that the estimates have already been established for all Sobolev norms with order of derivatives $k\le \ell-1$ so that $u_i(\cdot)\in L^{\infty}(0, T; W^{\ell-1,\infty})$ by Proposition \ref{pro:exactsol}. Then, consider a weak derivative $D^{\alpha}$ with $|\alpha|=\ell$. One has
\begin{equation*}
\partial_tD^{\alpha}(u_1-u_2)+D^{\alpha}(v(x)\cdot\nabla(u_1-u_2))=\nu \Delta(D^{\alpha}u_1-D^{\alpha}u_2)-D^{\alpha}[f(u_1)-f(u_2)].
\end{equation*}
One can the compute similarly that
\begin{equation*}
\begin{split}
\frac{d}{d t}\Vert D^{\alpha}(u_1-u_2)\Vert_p^p&=p\nu\int_{\Omega} |D^{\alpha}(u_1-u_2)|^{p-2}(D^{\alpha}(u_1-u_2))\Delta(D^{\alpha}(u_1-u_2))dx\\
&\quad -p\int_{\Omega} |D^{\alpha}(u_1-u_2)|^{p-2}(D^{\alpha}(u_1-u_2))D^{\alpha}[f(u_1)-f(u_2)]dx\\
&\quad -p\int_{\Omega} |D^{\alpha}(u_1-u_2)|^{p-2}(D^{\alpha}(u_1-u_2))D^{\alpha}(v\cdot\nabla(u_1-u_2))dx\\
&=: I_1'+I_2'+I_3'.
\end{split}
\end{equation*}
The dissipation term $I_1'$ can be treated similarly as above and it is found to be nonpositive:
\begin{multline*}
\int_{\Omega} |D^{\alpha}(u_1-u_2)|^{p-2}(D^{\alpha}(u_1-u_2))\Delta(D^{\alpha}(u_1-u_2))dx\\
=-(p-1)\int_{\Omega}|D^{\alpha}(u_1-u_2)|^{p-2}|\nabla(D^{\alpha}(u_1-u_2)|^2dx\le 0.
\end{multline*}

Now, we consider $I_2'$. By formula \eqref{equ:3-4}, one finds that
\[
D^{\alpha}(f(u_1)-f(u_2))
=\sum_{\beta\le \alpha} Q_{\beta}(u_i, D^{\alpha}u_i) D^{\beta}(u_1-u_2),
\]
where $Q_{\beta}(u_i, D^{\alpha}u_i)$
is some polynomial of $u_1, u_2$ and their derivatives. By Proposition \ref{pro:exactsol}, this term is thus bounded. Hence, by H\"older's inequality, one finds that
\begin{gather}\label{eq:stabilityaux1}
|I_2'|\le p \sum_{\beta\le \alpha} M_{\beta}\|D^{\alpha}(u_1-u_2)\|_p^{p-1}
\|D^{\beta}(u_1-u_2)\|_p.
\end{gather}

Now, consider $I_3'$. Note that
\[
D^{\alpha}(v\cdot\nabla(u_1-u_2))
=v\cdot \nabla D^{\alpha}(u_1-u_2)+[D^{\alpha}, v\cdot\nabla ](u_1-u_2),
\]
and
\[
\int_{\Omega} |D^{\alpha}(u_1-u_2)|^{p-2}(D^{\alpha}(u_1-u_2))v\cdot\nabla D^{\alpha}(u_1-u_2)dx=0,
\]
then one has
\begin{equation*}
\begin{split}
I_3'=&-p\int_{\T^d} |D^{\alpha}(u_1-u_2)|^{p-2}(D^{\alpha}(u_1-u_2))[D^{\alpha}, v\cdot\nabla](u_1-u_2)dx\\
=&-p\sum_{\beta<\alpha}C_{\alpha}^{\beta}\int_{\Omega} |D^{\alpha}(u_1-u_2)|^{p-2}(D^{\alpha}(u_1-u_2)) D^{\alpha-\beta}v\cdot \nabla(D^{\beta}(u_1-u_2)) \,dx.
\end{split}
\end{equation*}
Here $[\cdot,\cdot]$ represents Lie bracket, which means $[X,Y]=XY-YX$.\\
By Assumption \ref{ass:vassumption} and H\"older's inequality, one then has that
\begin{equation}\label{eq:stabilityaux2}
|I_3'|\le ~p\sum_{\beta<\alpha}C_{\alpha}^{\beta}\|D^{\alpha-\beta}v\|_{\infty}\Vert D^{\alpha}(u_1-u_2)\Vert_p^{p-1}  \|\nabla D^{\beta}(u_1-u_2)\|_p.
\end{equation}
Now the difference is that this term involves $\ell$-th order derivatives other than $\alpha$. Hence, we need to consider all the $\ell$-th order derivatives.

With the estimates \eqref{eq:stabilityaux1}--\eqref{eq:stabilityaux2} and by Young's inequality
\[
|ab| \le \frac{p-1}{p}|a|^{\frac{p}{p-1}}+\frac{1}{p}|b|^p,
\]
then by taking the sum for $\alpha$ with $|\alpha|=\ell$, one obtains that
\begin{equation*}
\begin{split}
\frac{d}{dt}\sum_{\alpha: |\alpha|=\ell}
\|D^{\alpha}(u_1-u_2)\|_p^p
\le &p M\sum_{|\alpha|=\ell}\sum_{\beta<\alpha}
\Vert D^{\alpha}(u_1-u_2)\Vert_p^{p-1}  \|\nabla D^{\beta}(u_1-u_2)\|_p\\
+&p M\sum_{|\alpha|=\ell}\sum_{\beta\le\alpha}
\Vert D^{\alpha}(u_1-u_2)\Vert_p^{p-1}  \|D^{\beta}(u_1-u_2)\|_p\\
\le &p M \sum_{|\alpha|=\ell}\Vert D^{\alpha}(u_1-u_2)\Vert_p^{p} 
+M\sum_{|\alpha'|<\ell}\Vert D^{\alpha'}(u_1-u_2)\Vert_p^{p},
\end{split}
\end{equation*}
where $M$ is a generic constant that may depend on $k$ and $d$, but is independent of $p$ and its concrete meaning can change from line to line.
Then, by Gr\"onwall's inequality and the induction assumption, one has
\begin{multline*}
\sum_{\alpha: |\alpha|=\ell}
\|D^{\alpha}(u_1(t)-u_2(t))\|_p^p
\le e^{pM t}\sum_{\alpha: |\alpha|=\ell}
\|D^{\alpha}(u_1(0)-u_2(0))\|_p^p\\
+\int_0^t M e^{pM(t-s)} e^{C_{\ell-1}p s}
\|u_1(0)-u_2(0)\|_{\ell-1,p}^p\,ds.
\end{multline*}
Hence, one eventually has
\begin{multline*}
\|u_1(t)-u_2(t)\|_{\ell,p}^p
\le e^{pM t}\sum_{\alpha: |\alpha|=\ell}
\|D^{\alpha}(u_1(0)-u_2(0))\|_p^p\\
+e^{p (\max(M, C_{\ell-1})+1)t}
\|u_1(0)-u_2(0)\|_{\ell-1,p}^p.
\end{multline*}
Letting $C_\ell:=\max(M, C_{\ell-1})+1$, one then has
\[
\|u_1(t)-u_2(t)\|_{k,p}^p \le e^{pC_\ell t}
\|u_1(0)-u_2(0)\|_{\ell,p}^p,
\]
Then,
\[
\|u_1(t)-u_2(t)\|_{\ell,p} \le e^{C_\ell t}
\|u_1(0)-u_2(0)\|_{\ell,p},
\]
where $C_\ell$ is independent of $p$.
Taking $p\to\infty$ then gives the result for $\|\cdot\|_{\ell,\infty}$ norm.
This then finishes the proof.
\end{proof}

\section{Error Estimates of the Random Splitting Method}\label{sec:error}

With the boundedness of Sobolev norms and stability above, we perform the error estimation for the random splitting method. Note that though the operators involved are not bounded, the actions on the sequence of numerical solutions are bounded due to the boundedness in Sobolev norms.
We first analyze the local truncation error,  and then establish the results about the bias and expected single-run error.

\subsection{The Local Truncation Errors}
In this subsection, we analyze the local truncation errors, for the preparation of our main results later.	
\begin{lemma}\label{lem:localtruncation}
Suppose Assumption \ref{ass:vassumption} holds. For all $p\in[2,+\infty]$ and $k\in\mathbb{N}$, and $a\in W^{k+6, p}\cap W^{k+5,\infty}$, there is a constant $C$ depending on the norm of $\|a\|_{k+6,p}+\|a\|_{k+5,\infty}$ such that
\begin{gather}
\sup_{\xi}\Vert S^{\xi}(\tau)[a]-T(\tau)[a]\Vert_{k,p} \leq C\tau^2,
\end{gather}
and that
\begin{equation}
\Vert S(\tau)a-T(\tau)a\Vert_{k,p} \leq C \tau^3.
\end{equation}
Here, we recall that $S^{\xi}(\tau)[a]=S_{\xi_3}(\tau)S_{\xi_2}(\tau)S_{\xi_1}(\tau)[a]$ for a permutation $\xi$ of $\{1,2,3\}$.
\end{lemma}
\begin{proof}
Consider the semigroup $S_1$ generated by $\cA$, one has by Taylor expansion that 
\begin{equation}\label{equ:tlexpandadv}
S_{1}(\tau)[a]=a+\tau \cA a+\frac{1}{2}\tau^2\cA^2a+R_{\cA},\quad
R_{\cA}[a]=\frac{1}{2}\int_{0}^{\tau}(\tau-s)^2\cA^3u(s)\,ds,
\end{equation}
with $u(s)=S_1(s)[a]$. 
By Proposition \ref{pro:exactsol}, $\|u(s)\|_{k+6,p}+\|u(s)\|_{k+5,\infty}$
bounded by a function of $\|a\|_{k+6,p}+\|a\|_{k+5,\infty}$.
Similarly, the heat semigroup is expanded as
\begin{equation}\label{eq:tlexpandheat}
S_{2}(\tau)[a]=a+\tau \cL a+\frac{1}{2}\tau^2\cL^2a+R_{\mathcal{L}},
\quad
R_{\mathcal{L}}[a]=\frac{1}{2}\int_{0}^{\tau}(\tau-s)^2\cL^3u(s)ds,
\end{equation}
with $u(s)=S_2(s)[a]$ here.

Consider the nonlinear semigroup $S_3(t)$. Letting $u(s)=S_3(s)[a]$, one has $\dot{u}=u-u^3$. Then, 
\begin{equation}
\frac{\partial^2u}{\partial t^2}=(1-3u^2)(u-u^3)=3u^5-4u^3+u,
\quad \frac{\partial^3u}{\partial t^3}=(15u^4-12u^2+1)(u-u^3).
\end{equation}
The Taylor expansion with integral remainder then gives
\begin{equation}\label{eq:tlexpandnonlinear}
S_{3}(\tau)a=a+\tau (a-a^3)+\frac{1}{2}\tau^2(1-3a^2)(a-a^3)+R_{\mathcal{R}}[a],
\end{equation}
where
\begin{equation}\label{equ:remainder-rrra}
R_{\mathcal{R}}[a]=\frac{1}{2}\int_{0}^{\tau}(\tau-s)^2h(u(s))ds,
\quad h(u)=(15u^4-12u^2+1)(u-u^3).
\end{equation}

At this point, we need to do the Taylor expansion of
$S^{\xi^n}[a]$. We take $S_1(\tau)S_3(\tau)S_2(\tau)[a]=S_{\mathcal{A}}(\tau)S_{\mathcal{R}}(\tau)S_{\mathcal{L}}(\tau)[a]$ as the example.
Using \eqref{equ:tlexpandadv}, one has
\begin{multline*}
S_{\mathcal{A}}(\tau)S_{\mathcal{R}}(\tau)S_{\mathcal{L}}(\tau)[a]
=S_{\mathcal{R}}(\tau)S_{\mathcal{L}}(\tau)[a]\\
+\tau \cA(S_{\mathcal{R}}(\tau)S_{\mathcal{L}}(\tau)[a])
+\frac{1}{2}\tau^2\cA^2(S_{\mathcal{R}}(\tau)S_{\mathcal{L}}(\tau)[a])
+R_{\cA}[S_{\mathcal{R}}(\tau)S_{\mathcal{L}}(\tau)[a]].
\end{multline*}
Applying Theorem \ref{theo:3-2}, $S_{\mathcal{R}}(\tau)S_{\mathcal{L}}(\tau)[a]$ is bounded in $W^{k+6,p}\cap W^{k+5,\infty}$ by a constant
depending on the norm of $a$. Hence, 
\[
\|R_{\cA}[S_{\mathcal{R}}(\tau)S_{\mathcal{L}}(\tau)[a]]\|_{k,p}
\le C\|S_{\mathcal{R}}(\tau)S_{\mathcal{L}}(\tau)[a]\|_{k+3,p}
\tau^3 \le C\tau^3.
\]
Moreover, using \eqref{eq:tlexpandnonlinear}, 
\begin{multline*}
S_{\mathcal{R}}(\tau)S_{\mathcal{L}}(\tau)[a]
=S_{\mathcal{L}}(\tau)[a]+\tau(S_{\mathcal{L}}(\tau)[a]-(S_{\mathcal{L}}(\tau)[a])^3)\\
+\frac{1}{2}\tau^2(1-3 (S_{\mathcal{L}}(\tau)[a])^2)
(S_{\mathcal{L}}(\tau)[a]-(S_{\mathcal{L}}(\tau)[a])^3)+R_{\cR}[S_{\mathcal{L}}(\tau)[a]].
\end{multline*}
It is easy to see that
\[
\|R_{\cR}[S_{\mathcal{L}}(\tau)[a]]\|_{k,p}\le C\tau^3.
\]
Moreover, for $\cA$ and $\cA^2$ terms, we do not need to expand
$S_{\mathcal{R}}(\tau)S_{\mathcal{L}}(\tau)[a]$ to third order, we only need the second order remainder and first order remainder. In particular, 
$\cA[(1-3 (S_{\mathcal{L}}(s)[a])^2)
(S_{\mathcal{L}}(s)[a]-(S_{\mathcal{L}}(s)[a])^3)]$
and $\cA^2(S_{\mathcal{L}}(s)[a]-(S_{\mathcal{L}}(s)[a])^3)$ are bounded uniformly in $W^{k,p}$. This then gives
\begin{multline*}
S_{\mathcal{A}}(\tau)S_{\mathcal{R}}(\tau)S_{\mathcal{L}}(\tau)[a]
=S_{\mathcal{L}}(\tau)[a]+\tau(S_{\mathcal{L}}(\tau)[a]-(S_{\mathcal{L}}(\tau)[a])^3)\\
+\tau \cA(S_{\cL}(\tau)[a])+\frac{1}{2}\tau^2(1-3 (S_{\mathcal{L}}(\tau)[a])^2)
(S_{\mathcal{L}}(\tau)[a]-(S_{\mathcal{L}}(\tau)[a])^3)\\
+\tau^2\cA(S_{\mathcal{L}}(\tau)[a]-(S_{\mathcal{L}}(\tau)[a])^3)
+\frac{1}{2}\tau^2\cA^2(S_{\cL}(\tau)[a])+R_1,
\end{multline*}
with
\[
\|R_1\|_{k,p}\le C\tau^3.
\]
We then apply \eqref{eq:tlexpandheat}. Similarly, we do the Taylor expansion of $S_{\cL}[a]=a+\int_0^{\tau}\cL u(s)\,ds$ only for those
in $O(\tau^2)$ and the one with second order remainder for those in $O(\tau)$. Then, 
\begin{multline}
S_{\mathcal{A}}(\tau)S_{\mathcal{R}}(\tau)S_{\mathcal{L}}(\tau)[a]
=a+\tau(\cL(a)+a-a^3+\cA(a))+\\
\frac{1}{2}\tau^2[\cL^2(a)+2\cA\cL(a)+2(1-3a^2)\cL(a)+(1-3 a^2)
(a-a^3)+\cA^2(a)+2\cA(a-a^3)]+R_2,
\end{multline}
with $\Vert R_2\Vert_{k,p}\leq C\tau^3$.

Other permutations can be treated similarly. The concrete formulas are listed in  Appendix \ref{sec:appendix}.
Note that we only need $\|a\|_{k+6,p}+\|a\|_{k+5,\infty}$. For example, in the expansion of $S_{2}(\tau)S_1(\tau)$, one has $\cL^3S_1(\tau)$ in the third order remainder term. Since  $S_1(\tau)[a]$ has no worse regularity than $a$, one needs only $6$ more derivatives (one does not need expand $S_1$ to get  $\tau^m\cL^3\cA^m$ and require $m+6$ more derivatives).  By these formulas, one has
\begin{multline}\label{equ:rec-s-tau-a}
S(\tau)[a]=\frac{1}{6}\sum_{\xi}S_{\xi_3}(\tau)S_{\xi_2}(\tau)S_{\xi_1}(\tau)[a]
=a+\tau(\cL a+\cA a+a-a^3)
+\frac{1}{2}\tau^2[(\cL+\cA)^2a\\
+(1-3a^2)(a-a^3)+(\cL+\cA)(a-a^3)+(1-3a^2)(\cL+\cA)a]+R_S,
\end{multline}
where $\xi$ ranges over the permutation of $\{1,2,3\}$ and 
\begin{equation}
\Vert R_{S}\Vert_{k,p}\leq C\tau^3,
\end{equation}
with $C$ depending on $a$ through the norm $\|a\|_{k+6,p}+\|a\|_{k+5,\infty}$.

Next, we consider the expansion of $T(\tau)$, which is a nonlinear semigroup involving the three operators together.
Using Duhamel's principle, one has
\begin{equation}\label{equ:t-tau-a-duha}
T(\tau)[a]=S_{\mathcal{L}+\mathcal{A}}(\tau)(a)+\int_{0}^{\tau}S_{\mathcal{L}+\mathcal{A}}(\tau-s)(u(s)-u^3(s))ds,
\end{equation}
where $u(s)$ is the solution to the PDE model with initial value $a$.
Then, one has
\begin{gather}\label{eq:uestimateaux1}
\sup_{s\le T}\left( \|u(s)\|_{k+6,p}+\|u(s)\|_{k+5,\infty}\right)<\infty,
\end{gather}
where the bound depends on $a$ through $\|a\|_{k+6,p}+\|a\|_{k+5,\infty}$.

For linear operator $\mathcal{L}+\mathcal{A}$, direct Taylor expansion gives
\begin{equation}\label{equ:sl+aa}
S_{\mathcal{L}+\mathcal{A}}(s_1)[w_0]=w_0+s_1(\cL+\cA)w_0+\frac{1}{2}s_1^2(\cL+\cA)^2w_0+R_{\mathcal{L}+\mathcal{A}}.
\end{equation}
The $W^{k,p}$ norm of the remainder can be bounded
by $Cs_1^3$ where $C$ is finite due to \eqref{eq:uestimateaux1}, if we take $w_0=a$ and $w_0=u(s)$. Moreover,
\begin{equation}\label{equ:tay-fu}
u(s)-u^3(s)=a-a^3+s(1-3a^2)(\cL a+\cA a+a-a^3)+R_{f},
\end{equation}
where $R_f=\int_{0}^{s}(s-s_1)G(s_1)ds_1$ with
\begin{multline*}
G(s_1)=-6u^2(s_1)(\mathcal{L}u(s_1)+\mathcal{A}u(s_1)+u(s_1)-u^3(s_1))^2\\
+(1-3u^2(s_1))(\mathcal{L}+\mathcal{A}+1-3u^2(s_1))(\mathcal{L}u(s_1)+\mathcal{A}u(s_1)+u(s_1)-u^3(s_1)).
\end{multline*}
Hence, eventually, one has
\begin{multline}\label{equ:rec-t-tau-a}
T(\tau)[a]=a+\tau(\cL a+\cA a+a-a^3)+\frac{1}{2}\tau^2\Big[(\cL+\cA)^2a\\
+(1-3a^2)(\cL a+\cA a+a-a^3)+(\cL+\cA)(a-a^3)\Big]+R_T.
\end{multline}
The remainder term is bounded by $C\tau^3$ again in $W^{k,p}$ norm and we omit the details. Here, $C$ depends on $a$ through the norm $\|a\|_{k+6,p}+\|a\|_{k+5,\infty}$.

With these expansions, the claims of the local truncation error hold.
\end{proof}

\begin{corollary}
Consider problem \eqref{equ:acdrift} and assume that $v(x)$ satisfies Assumption \ref{ass:vassumption}. Let $p\in[2,+\infty]$, $k\in\mathbb{N}$ and assume that $u_0\in W^{k+6,p}\cap W^{k+5,\infty}$. Let $u(t)$ be the exact solution to \eqref{equ:acdrift}, and $u_n$ be the numerical solution in definition \eqref{rela-rec} for the random splitting method with time step $\tau>0$.  Then, there exists $C>0$ such that
\begin{equation}
\sup_{\b{\xi}}\Vert u_n-u(t_n)\Vert_{k,p}\leq C\tau.
\end{equation}
\end{corollary}
\begin{proof}
According to Theorem \ref{theo:3-2},
\begin{gather}\label{auxeq:auxnorms}
\sup_{\b{\xi}}\sup_{n: n\tau\le T}\left( \|u_n\|_{k+6,p}+\|u_{n}\|_{k+5,\infty}\right)<\infty.
\end{gather}
By Theorem \ref{theo:3-3}, one then has
\begin{gather*}
\left\Vert T(\tau)[u_{n-1}]-T(\tau)[u(t_{n-1})] \right\Vert_{k,p}\leq e^{C\tau}\Vert u_{n-1}-u(t_{n-1})\Vert_{k,p},
\end{gather*}
where $C$ depends on the $\|u_{n}\|_{\max(k-1, 0),\infty}$ norms, which are uniformly bounded. 

Now, we take $a=u_{n-1}$ which is bounded by \eqref{auxeq:auxnorms}. Applying Lemma \ref{lem:localtruncation}, there exists $C>0$, such that
\begin{equation*}
\Vert S^{\xi^n}(\tau)u_{n-1}-T(\tau)u_{n-1}\Vert_{k,p}\leq C\tau^2.
\end{equation*}
Then,
\begin{equation*}
\begin{split}
\Vert u_n-u(t_n)\Vert_{k,p}\leq & \Vert S^{\xi^n}(\tau)[u_{n-1}]-T(\tau)[u_{n-1}]\Vert_{k,p}+\Vert T(\tau)[u_{n-1}]-T(\tau)[u(t_{n-1}])\Vert_{k,p}\\
\leq  &C\tau^2+e^{C\tau}\Vert u_{n-1}-u(t_{n-1})\Vert_{k,p}.
\end{split}
\end{equation*}
This then gives the desired result.
\end{proof}

\subsection{Convergence analysis: expectation of the error}

In this subsection, we establish the estimates of the expected single-run error, defined by
\begin{gather}
\mathcal{E}_{k,p}:=\E_{\b{\xi}}\|u_n-u(\cdot, t_n)\|_{k,p}.
\end{gather}
In particular, we have the following estimate about the single run error.
\begin{theorem}\label{theo:4-2}
Let $T>0$. Consider problem \eqref{equ:acdrift} and assume that $v(x)$ satisfies Assumption \ref{ass:vassumption}. Let $p\in[2,+\infty)$, $k\in\mathbb{N}$ and assume that $u_0\in W^{k+6,p}\cap W^{k+5,\infty}$. Let $u(t)$ be the exact solution to \eqref{equ:acdrift}, and $u_n$ be the numerical solution in \eqref{rela-rec} for the random splitting method with time step $\tau\le T$, with the same initial value.  Then, there exists $C>0$ such that for any $t_n\le T$, one has
\begin{equation}
\E_{\b{\xi}}\|u_n-u(\cdot, t_n)\|_{k,p} \leq \left(\E_{\b{\xi}}\Vert u_{n}-u(\cdot,t_n)\Vert_{k,p}^{p}\right)^{1/p} \le C \tau^{3/2}.
\end{equation}
\end{theorem}
\begin{proof}
By Proposition \ref{pro:exactsol} and Theorem \ref{theo:3-2},
\begin{gather}\label{eq:aux2}
\sup_{s\le T}\left( \|u(s)\|_{k+6,p}+\|u(s)\|_{k+5,\infty}\right)<\infty,
\quad \sup_{\b{\xi}}\sup_{n: n\tau\le T}\left( \|u_n\|_{k+6,p}+\|u_{n}\|_{k+5,\infty}\right)<\infty.
\end{gather}

First consider $p\in [2,\infty)$. Define
\begin{gather}
A_n:=\E_{\b{\xi}}\|u_n-u(\cdot, t_n)\|_{k,p}^p
=\sum_{|\alpha|\le k}\E_{\b{\xi}}\|D^{\alpha}u_n-D^{\alpha}u(\cdot, t_n)\|_p^p.
\end{gather}
Using the fact
\begin{multline}
D^{\alpha}u_{n+1}-D^{\alpha}u(\cdot, t_{n+1})
=D^{\alpha}(S^{\xi^n}(\tau)[u_n]-T(\tau)[u(\cdot, t_n)])\\
=D^{\alpha}(S^{\xi^n}(\tau)[u_n]-T(\tau)[u_n])
+D^{\alpha}(T(\tau)[u_n]-T(\tau)[u(\cdot, t_n)]).
\end{multline}
Denote
\[
a:=D^{\alpha}(T(\tau)[u_n]-T(\tau)[u(\cdot, t_n)]), 
\quad b:=D^{\alpha}(S^{\xi^n}(\tau)[u_n]-T(\tau)[u_n]).
\]
We consider
\[
|a+ b|^p=|a|^p+p|a|^{p-2}a b+\int_0^1(1-s)p(p-1)|a+s  b|^{p-2} b^2\,ds.
\]
Taking expectation on $\xi^n$ first and applying Lemma \ref{lem:localtruncation}, one has
\[
\E_{\xi^n}|a+ b|^p
\le |a|^p+p|a|^{p-2}a\E_{\xi^n}b+Cp(p-1)2^{p-1}(|a|^{p-2}\tau^4+C^p \tau^{2p}).
\]
With a constant $\bar{C}$ {\it that depends on $p$}, one has
\[
\E_{\xi^n}|a+ b|^p
\le |a|^p+\bar{C} |a|^{p-1}\tau^3+\bar{C}(\tau |a|^{p}+\tau^{\frac{3p}{2}+1}+ \tau^{2p})
\]
by Young's inequality. Moreover, by the stability of the model in Theorem \ref{theo:3-3},
\[
\|a\|_p^p\le e^{pC\tau}\|D^{\alpha}u_n-D^{\alpha}u(\cdot, t_n)\|_p^p.
\]
Integrating on $\Omega$, taking the full expectation and summing over $\alpha$, with the fact $\tau\le T$, one has
\[
A_{n+1}\le e^{pC\tau}(1+\bar{C}\tau)A_n+\bar{C}\tau^{\frac{3p}{2}+1}.
\]
This then gives
\[
A_n\le \bar{C}\tau^{\frac{3p}{2}},
\]
giving the desired result.
\end{proof}

\begin{remark}
Note that the proof above cannot be generalized to $p=\infty$. The reason is that it is hard to close the estimate of $\int_0^1(1-s)p(p-1)|a+s  b|^{p-2} b^2\,ds$ that does not blow up
as $p\to\infty$.
\end{remark}

\subsection{Convergence analysis: bias}

In this subsection, we establish estimate of the bias. Recall in \eqref{def:v-n}, $v_n$ is the expectation of the numerical solution $u_n$.
Then, the bias is $v_n-u(\cdot, t_n)$. We will establish the bias under $W^{k,p}$ norms defined by
\begin{gather}
\mathcal{B}_{k,p}:=\|v_n-u(\cdot, t_n)\|_{k,p}.
\end{gather}

\begin{theorem}\label{theo:4-3}
Let $T>0$. Consider problem \eqref{equ:acdrift} and assume that $v(x)$ satisfies Assumption \ref{ass:vassumption}. Let $p\in[2,+\infty]$, $k\in\mathbb{N}$ and assume that $u_0\in W^{k+6,p}\cap W^{k+5,\infty}$. Let $u(t)$ be the exact solution to \eqref{equ:acdrift}, and $u_n$ be the numerical solution in \eqref{rela-rec} for the random splitting method with time step $\tau\le T$, with the same initial value. Recall $v_n$ defined in \eqref{def:v-n}. Then, there exists $C>0$ such that for any $t_n\le T$, one has
\begin{equation}
\Vert v_{n}-u(\cdot, t_n)\Vert_{k,p}\leq C\tau^2.
\end{equation}
\end{theorem}
\begin{proof}
Note that 
\[
u_{n+1}=S^{\xi_n}(\tau)[u_n],\quad v_{n+1}=\E_{\b{\xi}}[\E[S^{\xi_n}(\tau)[u_n]| \xi_1,\cdots,\xi_{n-1}]]=\E_{\b{\xi}}S(\tau)[u_n],
\]
where $S$ is defined in \eqref{equ:ex-random-split-term}. However, since $S$ is nonlinear, we need to decompose
\[
\begin{split}
v_{n+1}-u(\cdot, t_{n+1})&=\E_{\b{\xi}}S(\tau)[u_n]-T(\tau)[u(\cdot, t_n)]\\
&=(\E_{\b{\xi}}S(\tau)[u_n]-S(\tau)[v_n])+(S(\tau)[v_n]-T(\tau)[u(\cdot, t_n)]).
\end{split}
\]
As in the proof of Theorem \ref{theo:4-2},
\begin{gather}\label{eq:aux3}
\sup_{s\le T}\left( \|u(s)\|_{k+6,p}+\|u(s)\|_{k+5,\infty}\right)<\infty,
\quad \sup_{\b{\xi}}\sup_{n: n\tau\le T}\left( \|u_n\|_{k+6,p}+\|u_{n}\|_{k+5,\infty}\right)<\infty.
\end{gather}
Consequently, $v_n$ has the same control by Jensen's inequality. Then, one has by the local truncation error in Lemma \ref{lem:localtruncation} and  by  the stability of the model in Theorem \ref{theo:3-3} that
\begin{gather*}
\begin{split}
\|S(\tau)[v_n]-T(\tau)[u(\cdot, t_n)]\|_{k,p}
&\le \|S(\tau)[v_n]-T(\tau)[v_n]\|_{k,p}+\|T(\tau)[v_n]-T(\tau)[u(\cdot, t_n)]\|_{k,p}\\
&\le C\tau^3+e^{C\tau}\|v_n-u(\cdot, t_n)\|_{k,p}.
\end{split}
\end{gather*}
Here, the constant $C$ does not blow up as $p\to\infty$.

Next, we consider $\E_{\b{\xi}}S(\tau)[u_n]-S(\tau)[v_n]$. By the semigroup expansion \eqref{equ:rec-s-tau-a}, we get:
\begin{multline}
\mathbb{E}_{\b{\xi}}(S(\tau)[u_{n}])-S(\tau)[\mathbb{E}_{\b{\xi}}u_{n}]
=-\tau [\E_{\b{\xi}}(u_{n}^3)-(\mathbb{E}_{\b{\xi}}u_{n})^3]-2\tau^2 [\mathbb{E}_{\b{\xi}}(u_{n}^3)-(\mathbb{E}_{\b{\xi}}u_{n})^3]\\
+\frac{3}{2}\tau^2 [\mathbb{E}_{\b{\xi}}(u_{n}^5)-(\mathbb{E}_{\b{\xi}}u_{n})^5]
-\frac{1}{2}\tau^2(\cL+\cA)[\mathbb{E}_{\b{\xi}}(u_{n}^3)-(\mathbb{E}_{\b{\xi}}u_{n})^3]\\
-\frac{3}{2}\tau^2[\mathbb{E}_{\b{\xi}}(u_{n}^2(\cL+\cA)u_{n})-(\mathbb{E}_{\b{\xi}}u_{n})^2(\cL+\cA)(\mathbb{E}_{\b{\xi}}u_{n})]+\bar{R},
\end{multline}
where $\bar{R}$ is $O(\tau^3)$ in the sense of $W^{k,p}$ norm.

As one can see, the differences are various bias brought by nonlinearity.
In general, 
\begin{multline*}
f(u_n)-f(\E_{\b{\xi}}u_n)
=f'(\E_{\b{\xi}}u_n)(u_n-\E_{\b{\xi}}u_n)\\
+\int_0^1(1-s)f''(\E_{\b{\xi}}u_n+s(u_n-\E_{\b{\xi}}u_n))(u_n-\E_{\b{\xi}}u_n)^2\,ds.
\end{multline*}
Taking expectation, the first term vanishes. 
By \eqref{eq:aux3} and Lemma \ref{lem:localtruncation}, the second term 
should be $C\tau^2$ in $W^{k,p}$. Hence, one has
\[
\|\mathbb{E}_{\b{\xi}}(S(\tau)[u_{n}])-S(\tau)[\mathbb{E}_{\b{\xi}}u_{n}]\|_{k,p}
\le C\tau^3.
\]

Hence,
\[
\|v_{n+1}-u(\cdot, t_{n+1})\|_{k,p}
\le e^{C\tau}\|v_{n}-u(\cdot, t_{n})\|_{k,p}+C\tau^3.
\]
Note that the constants $C$ in all the estimates here can be generalized to $p=\infty$. Hence, the result above is also valid for $p=\infty$.
The result then follows. 
\end{proof}

\subsection{Discussions}

We now have proved the convergence of the random splitting method
rigorously for the model \eqref{equ:acdrift}, which is a practical model for phase separation in a background flow that involves unbounded and nonlinear operators. The key is to get the stability of the model in the same Sobolev space and the propagation of the initial regularity 
so that the local truncation errors involving the unbounded operators
can be estimated. The nonlinearity brings extra bias that needs to be estimated.

Here, we provide an intuitive understanding of the order of convergence. 
In fact, the random splitting method converges through a law of large number mechanism in time, due to the averaging effect, as explained for the random batch method in \cite{jin2020random,jin2022random,jin2021convergence,jin2022}.
For a model consisting of $p$ operators, with the random splitting strategy used in this paper, the number of computational intervals is $N=O(\tau^{-1})$. 
Among this number of random permutations, a complete set of random permutations consisting of each of the $p!$ ones would be considered as an unbiased evolution. With the natural central limit theorem, the fluctuation that leads to the biased evolution, whose number is of order
$O(\sqrt{N})=O(\tau^{-1/2})$. 
As mentioned in Lemma \ref{lem:localtruncation}, each biased step contributes local error of order $O(\tau^2)$. Hence, the total error would be of order
\[
O(\tau^{-1/2})*O(\tau^2)=O(\tau^{3/2}).
\]
In the random splitting here, the random permutations are independent within each step. If one does the permutation with correlation such that each permutation appears exactly once during the consecutive $p!$ steps, then the error is expected to be lower. However, due to the correlation, the analysis would be more involved. The limit behavior as $p$ goes large is also unclear. This could be an interesting topic for future study.

Lastly, we remark that one may use both the single-run result or the averaged result to approximate the true solution.
A single run typically gives an $O(\tau^{3/2})$ error.
Although the averaged solution has an error of order $O(\tau^2)$, to get a good estimate of the expectation, one needs the number of solutions $M$ large so that
the fluctuation is of the same order. In particular, we need
$O(\mathrm{Variance}/M)=O(\tau^4)$. The variance could be estimated
by the square of the single run error, which is $O(\tau^3)$. Hence, to get the desired accuracy, one needs $M=O(\tau^{-1})$, which seems more expensive
than doing the symmetric Trotter splitting.
This is in contrast to the recently proposed random ordinate method for radiative transfer equations \cite{li2024random}, where the bias is the square of the single-run error, for which case doing independent samples and taking the average would not enlarge the complexity while more samples would reduce the low resolution issue in that problem.

\section{Numerical Experiments}\label{sec:numerical}

In this section, we validate our theoretic results in Theorem \ref{theo:4-2} and Theorem \ref{theo:4-3} by an illustrating example.

The reference solution will be computed by an exponential Runge-Kutta method to overcome the stiffness in diffusion. The spatial discretization will be performed by the Fourier spectral method. 
In particular, on the Fourier side, one could find that the Fourier transform of the exact solution satisfies
\begin{multline}\label{equ:exact}
\hat{u}(t,k)=e^{-\frac{4\pi^2}{L^2}|k|^2\nu t}\hat{u}(0,k)\\
-e^{-\frac{4\pi^2}{L^2}|k|^2\nu t}\int_{0}^{t}e^{\frac{4\pi^2}{L^2}|k|^2\nu s}\left[\widehat{f(u)}(s,k)+i\pi\frac{2k}{L}\cdot \widehat{(v u)}(s,k)\right]\,ds,
\end{multline}
where $\hat{\phi}$ means the Fourier transform (Fourier series) of the periodic function $\phi$ and $k$ is the frequency which is a $d$-dimensional integer. Hence, the following midpoint formula could be used for the reference solution (by setting $t=\tau$ and using $u(\tau/2)$ for the integrand to approximate, while $u(\tau/2)$ itself is approximated by an Euler type approximation).
\begin{equation}\label{equ:mid-point}
\begin{split}
&\hat{u}_{n+1/2}\left(k\right)=e^{-\frac{2\pi^2}{L^2}|k|^2\nu \tau}\hat{u}_n(k)-\frac{L^2}{4\pi^2|k|^2\nu}\left(1-\exp(-\frac{2\pi^2}{L^2}|k|^2\nu\tau)\right)g(u_n, k),\\
&\hat{u}_{n+1}(k)= e^{-\frac{4\pi^2}{L^2}|k|^2\nu \tau}\hat{u}_n(k)-\frac{L^2}{4\pi^2|k|^2\nu}\left(1-\exp(-\frac{4\pi^2}{L^2}|k|^2\nu\tau)\right)g(u_{n+1/2},k),
\end{split}
\end{equation}
where
\begin{equation}
g(u, k)=\widehat{f(u)}(k)+i\pi\frac{2 k}{L}\cdot\widehat{(v u)}(k).
\end{equation}
This is clearly a second order scheme. Here, the case $k=0$ needs to be treated separately as the $k\to 0$ limit. For spatial discretization, we use the discrete Fourier transform and $k$ then ranges from $-N/2+1$ to $N/2$, which is then implemented by the fast Fourier transform.

For the implementation of Algorithm \ref{alg:randomsplit}, we again take the Fourier spectral method for spatial discretization, which enjoys the spectral accuracy so that we may focus on the order of time discretization. 
For example, the heat semigroup can be solved exactly in time due to the observation
\begin{equation*}
\frac{d}{dt}\hat{u}(t,k)=-\frac{4\nu \pi^2}{L^2}|k|^2\hat{u}(t,k),
\end{equation*}
\begin{equation*}
\hat{u}(\tau,k)=\hat{u}(0,k)\exp\left(-\frac{4\nu \pi^2}{L^2}|k|^2\tau\right),
\end{equation*}
so that there is no stiffness. The advection term is similarly expressed as 
\begin{equation*}
\frac{d}{dt}\hat{u}(k)=-i\frac{2\pi}{L}k\cdot\widehat{(v u)}(k),
\end{equation*}
which is solved by an explicit four-stage Runge-Kutta method (RK4-method).
The nonlinear semigroup is solved by the ordinary differential equation, and is given by \eqref{eq:nonlinearsemigroup}.

In this experiment, we take $d=2$ and the background is the following shear flow
\begin{gather}
v(x,y)=\begin{pmatrix}-0.75\sin y \\ 0\end{pmatrix}.
\end{gather}
The parameters in problem \eqref{equ:acdrift} are taken to be
\begin{equation}
L=2\pi, \nu=1,
\end{equation}
and initial value is given by
\begin{gather}
 u_0(x,y)=1+0.5\sin x+\exp(0.7\sin y).
\end{gather}
We fix the spatial step 
\[
\Delta x=\Delta y=h:=\pi\times 2^{-7}.
\] 
The terminal time is chosen to be $T=1$, and the reference solution is computed using the time step $\tau=2^{-14}$.
For the order of bias and expected single-run error, we choose five different time steps $\tau=2^{-m}$ with $m=4:8$.

We test the errors in $L^2$ and $W^{1,2}$ norms, and repeat the method for $N_E=10^4$ times. 
Let $u_{n,ij}^{\ell}$ be the numerical solution at location $(x_i, y_j)=(ih, jh)$ and time $t_n$ for the $\ell$-th experiment. Define the error function
\begin{gather}
\varepsilon_{n,ij}^{\ell}:=u_{n,ij}^{\ell}-u(x_i,y_j, t_n).
\end{gather}
 The expected $\ell^2$-error and the bias in $\ell^2$ are then approximated by
\begin{gather}
\begin{split}
&\mathcal{E}_2^h:=\max_{0\le n\le N}\frac{1}{N_E}\sum_{\ell=1}^{N_E} \sqrt{\sum_{i,j}|\varepsilon_{n,ij}^{\ell}|^2 \Delta x \Delta y},\\
&\mathcal{B}_2^h:=\max_{0\le n\le N}\sqrt{\sum_{i,j}\left|\frac{1}{N_E}\sum_{\ell=1}^{N_E}\varepsilon_{n,ij}^{\ell} \right|^2 \Delta x \Delta y}.
\end{split}
\end{gather}

For the $W^{1,2}$ norm, we use the Fourier spectral method to approximate the derivative, i.e., for function $u$ the derivative is given by
\[
Du:=w, \quad \hat{w}(k)=i\frac{2\pi}{L}k \hat{u}(k).
\]
Note that $Dw$ is a vector value. Then the error and bias in $W^{1,2}$ norm are approximated by
\begin{gather}
\begin{split}
& \mathcal{E}_{1,2}^h:=\max_{0\le n\le N}\frac{1}{N_E}\sum_{\ell=1}^{N_E} \sqrt{\sum_{i,j}(|\varepsilon_{n,ij}^{\ell}|^2+|D\varepsilon_{n,ij}^{\ell}|^2) \Delta x \Delta y},\\
& \mathcal{B}_{1,2}^h:=\max_{0\le n\le N}\sqrt{\sum_{i,j}
\left(\left|\frac{1}{N_E}\sum_{\ell=1}^{N_E}\varepsilon_{n,ij}^{\ell} \right|^2
+\left|\frac{1}{N_E}\sum_{\ell=1}^{N_E}D\varepsilon_{n,ij}^{\ell} \right|^2\right) \Delta x \Delta y}.
\end{split}
\end{gather}
We remark that the $L^2$ and $W^{1,2}$ norms can also be computed using the Fourier transform in the Fourier space by the Plancherel theorem (Parseval's identity).

\begin{figure}[!ht]
\centering
\includegraphics[width=0.45\textwidth]{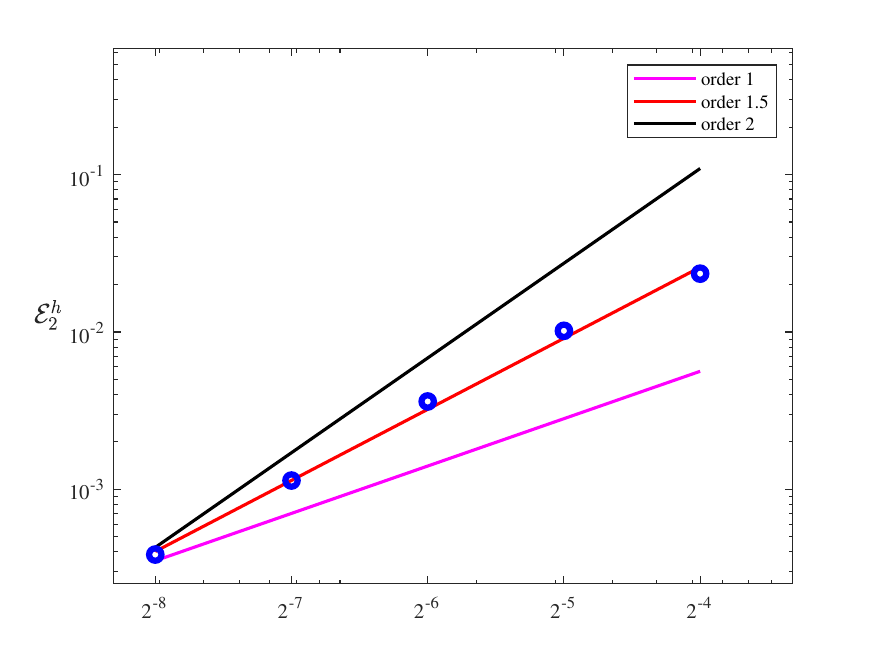}
\includegraphics[width=0.45\textwidth]{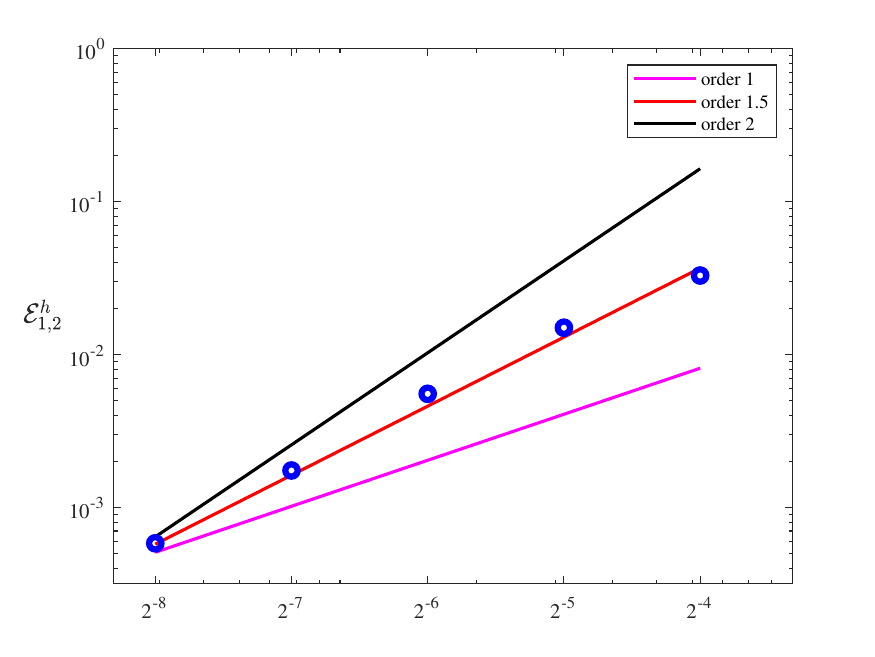}
\caption{The convergence order of expectation of error, for $\tau=2^{-4},2^{-5},\cdots,2^{-8}$,$N_E=10^4$. (a) $L^2$ norm; (b) $W^{1,2}$ norm.}
\label{fig:exoer}
\end{figure}

\begin{figure}[!ht]
\centering
\includegraphics[width=0.45\textwidth]{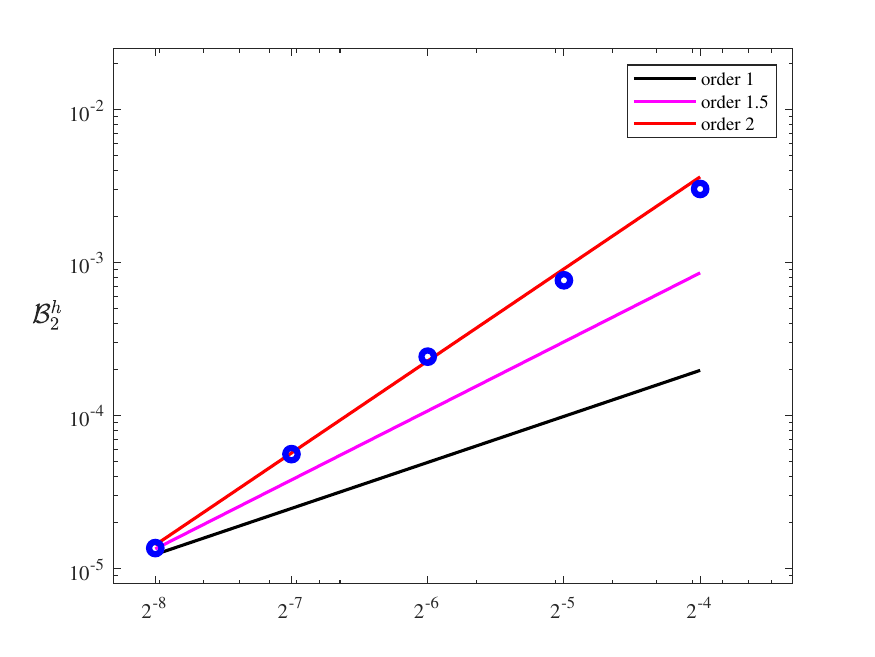}
\includegraphics[width=0.45\textwidth]{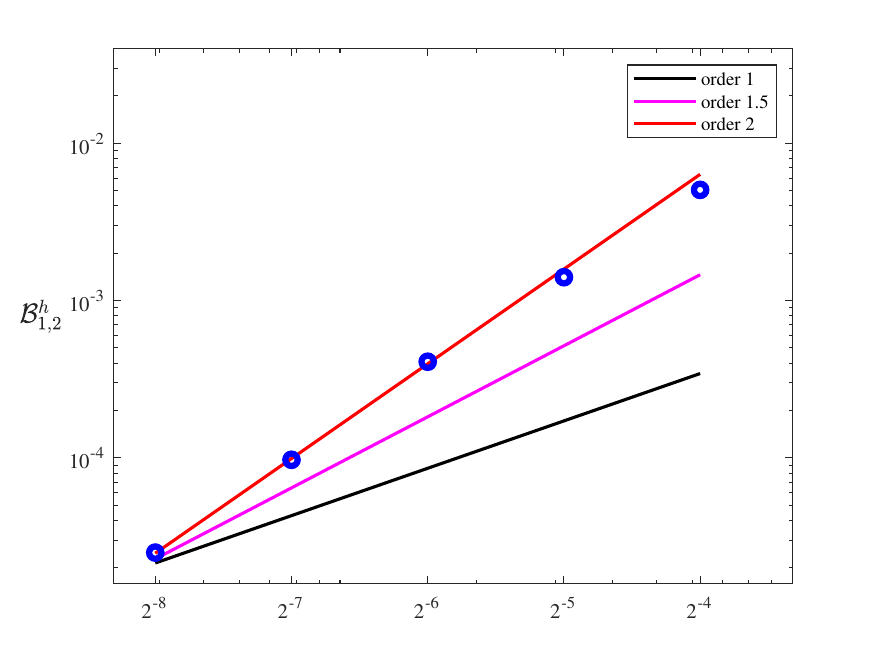}
\caption{The convergence order of bias, for $\tau=2^{-4},2^{-5},\cdots,2^{-8}$,$N_E=10^4$. (a) $L^2$ norm; (b) $W^{1,2}$ norm.}
\label{fig:eroex}
\end{figure}

Figure \ref{fig:exoer} shows the order of expected single-run error versus the time step. According to the numerical results, the error converges to order 1.5, agreeing with the theory. Note that for any fixed order of splitting, the convergence is only first order. This shows that the randomized method enjoys better convergence rate with the same cost.
Figure \ref{fig:eroex} shows the order of bias versus the time step is two, also consistent with the theory. %Comparing with the expected single-run error, the bias has higher order. This shows that in many problems, we can calculate the approximate solution by averaging the numerical solutions of the random splitting method.

\subsection*{Acknowledgement}
This work is partially supported by the National Key R\&D Program of China under grant No. 2020YFA0712000. The work of L. Li is also partially supported by NSFC 12371400 and 12031013, Shanghai Municipal Science and Technology Major Project 2021SHZDZX0102.

\bibliographystyle{plain}
\bibliography{references}

\appendix

\section{Various expressions for the expansions of semigroups}\label{sec:appendix}
Formulas in the proof of Lemma \ref{lem:localtruncation}.
\begin{multline}
S_{\mathcal{A}}(\tau)S_{\mathcal{R}}(\tau)S_{\mathcal{L}}(\tau)[a]=a+\tau (\cL a+\cA a+a-a^3)
+\frac{1}{2}\tau^2\Big[\cL^2a+\cA^2a\\
+(1-3a^2)(a-a^3)+2(1-3a^2)\cL a+2\cA(a-a^3)+2\cA\cL a\Big]+R_{\cA\cR\cL}.
\end{multline}

\begin{multline}
S_\cL(\tau)S_\cR(\tau)S_\cA(\tau)[a]=a+\tau(\cL a+\cA a+a-a^3)
+\frac{1}{2}\tau^2\Big[\cL^2a+\cA^2a\\
+(1-3a^2)(a-a^3)+2(1-3a^2)\cA a+2\cL(a-a^3)+2\cL\cA a\Big]+R_{\cL\cR\cA}
\end{multline}

\begin{multline}
S_\cA(\tau)S_\cL(\tau)S_\cR(\tau)[a]=a+\tau(\cL a+\cA a+a-a^3)
+\frac{1}{2}\tau^2\Big[\cL^2a+\cA^2a\\
+(1-3a^2)(a-a^3)+2\cL (a-a^3)+2\cA (a-a^3)+2\cA\cL a \Big]+R_{\cA\cL\cR}
\end{multline}

\begin{multline}
S_\cL(\tau)S_\cA(\tau)S_\cR(\tau)[a]=a+\tau(\cL a+\cA a+a-a^3)
+\frac{1}{2}\tau^2\Big[\cL^2a+\cA^2a\\
+(1-3a^2)(a-a^3)+2\cL (a-a^3)+2\cA (a-a^3)+2\cL\cA a\Big]+R_{\cL\cA\cR}
\end{multline}
\begin{multline}
S_\cR(\tau)S_\cA(\tau)S_\cL(\tau)[a]=a+\tau(\cL a+\cA a+a-a^3)+\frac{1}{2}\tau^2\Big[\cL^2a+\cA^2a\\
+(1-3a^2)(a-a^3)+2\cA\cL a+2(1-3a^2)\cL a+2(1-3a^2)\cA a\Big]+R_{\cR\cA\cL}
\end{multline}
\begin{multline}
S_\cR(\tau)S_\cL(\tau)S_\cA(\tau)[a]=a+\tau(\cL a+\cA a+a-a^3)
+\frac{1}{2}\tau^2[\cL^2a+\cA^2a\\
+(1-3a^2)(a-a^3)+2\cL\cA a+2(1-3a^2)\cL a+2(1-3a^2)\cA a]+R_{\cR\cL\cA}
\end{multline}
With the conditions given, it holds that
\begin{equation}
\max\{\Vert R_{\cA\cR\cL}\Vert_{k,p},\Vert R_{\cL\cR\cA}\Vert_{k,p},\Vert R_{\cA\cL\cR}\Vert_{k,p},\Vert R_{\cL\cA\cR}\Vert_{k,p},\Vert R_{\cR\cA\cL}\Vert_{k,p},\Vert R_{\cR\cL\cA}\Vert_{k,p}\}\leq  C\tau^3
\end{equation}
where the constant $C$ depends on the $W^{k+6,p}$ and $W^{k+5,\infty}$ norms of $a$.

\end{document}